\documentclass[10pt]{amsart}
\usepackage{amsmath, amssymb}
\usepackage{mathrsfs}
\usepackage{mathabx}
\newcommand{\no}[1]{#1}
\renewcommand{\no}[1]{}
\no{\usepackage{times}\usepackage[subscriptcorrection, slantedGreek, nofontinfo]{mtpro}
\renewcommand{\Delta}{\upDelta}}
\usepackage{color}

\date{\today}
\newtheorem{theorem}{Theorem}[section]
\newtheorem{proposition}{Proposition}[section]
\newtheorem{lemma}{Lemma}[section]

\newtheorem{corollary}{Corollary}[section]

\theoremstyle{remark}
\newtheorem{remark}{Remark}[section]

\numberwithin{equation}{section}

\title[H\"older stability]{H\"older stability in determining the potential and the damping coefficient in a wave equation}

\author[Ka\"{\i}s Ammari]{Ka\"{\i}s Ammari}
\address{Ka\"{\i}s Ammari, UR Analysis and Control of PDEs, UR 13ES64, Department of Mathematics, Faculty of Sciences of Monastir, University of Monastir, 5019 Monastir, Tunisia }
\email{kais.ammari@fsm.rnu.tn}

\author[Mourad Choulli]{Mourad Choulli}
\address{Mourad Choulli, Institut \'Elie Cartan de Lorraine, UMR CNRS 7502, Universit\'e de Lorraine, Boulevard des Aiguillettes, BP 70239, 54506 Vandoeuvre les Nancy cedex - Ile du Saulcy, 57045 Metz cedex 01, France}
\email{mourad.choulli@univ-lorraine.fr}

\author[Faouzi Triki]{Faouzi Triki$^\dag$}
\address{Laboratoire Jean Kuntzmann,  UMR CNRS 5224, 
Universit\'e  Grenoble-Alpes, 700 Avenue Centrale,
38401 Saint-Martin-d'H\`eres, France}
\email{Faouzi.Triki@imag.fr}

\date{}

\begin{document}

\begin{abstract}
We improve the preceding results obtained by the first and the second
authors in \cite{AC1}. They concern the stability issue of the inverse
problem that consists in determining the potential and the damping
coefficient in a wave equation from an initial-to-boundary
operator. We partially modify the arguments in \cite{AC1} to show that actually we have H\"older stability instead of logarithmic stability.

\end{abstract}


\keywords{inverse problem, H\"older stability, wave equation, damping coefficient, potential. \\ 
$^\dag$ FT is partially supported by Labex PERSYVAL-Lab (ANR-11-LABX-0025-01)}

\maketitle


\section{Introduction}\label{section1}

Let $M=(M,g)$ be a compact $n$-dimensional Riemannian manifold with boundary. Recall that, in local coordinates $x=(x^1,\ldots ,x^n)$,
\[
g=g_{ij}dx^i\otimes dx^j.
\]
Here and henceforth, we adopt the Einstein convention summation for repeated indices. 

\smallskip
For two vector fields $X$ and $Y$ over $M$,
\[
\langle X,Y\rangle = g_{ij}X^iY^j,
\]
when $X=X^i\partial _i$ and $Y=Y^j\partial _j$, where $(\partial _1,\ldots ,\partial _n)$ is the dual basis to $(x^1,\ldots ,x^n)$. Set then $|X|=\sqrt{\langle X,X\rangle}$.

\smallskip
Recall that the gradient of $u\in C^\infty (M)$ is the vector field given by
\[
\nabla u= g^{ij}\partial _iu\partial _j 
\]
and the Laplace-Beltrami operator is the operator acting as follows
\[
\Delta u=|g|^{-1/2}\partial _i\left( |g|^{1/2}g^{ij}\partial _ju\right),
\]
where $(g^{ij})$ is the inverse of the metric $g$ and $|g|$ is the determinant of $g$.


\smallskip
Consider the following initial-boundary value problem, abbreviated to IBVP in the sequel, for the wave equation:
\begin{equation}\label{e1}
\left\{
\begin{array}{lll}
 \partial _t^2 u - \Delta u + q(x)u + a(x) \partial_t u = 0 \;\; &\mbox{in}\;   Q=\ring{M} \times (0,\tau), 
 \\
u = 0 &\mbox{on}\;  B =\partial M \times (0,\tau), 
\\
u(\cdot ,0) = u_0,\; \partial_t u (\cdot ,0) = u_1.
\end{array}
\right.
\end{equation}

Let $\mathcal{H}=H_0^1(\ring{M}) \oplus L^2(\ring{M})$. The analysis carried out in \cite[sections 5 and 6, Chapter XVIII]{DL} (see also \cite[Chapter 2] {BY}) enables us to deduce that, for any $q,a\in L^\infty (\Omega  )$, $\tau >0$ and $(u_0,u_1)\in \mathcal{H}$, the IBVP \eqref{e1} has a unique solution \[ u:=\mathcal{S}_{q,a}(u_0,u_1)\in C([0,\tau ],H_0^1(\ring{M}  ))\] so that $\partial _tu\in C([0,\tau ],L^2(\ring{M}  ))$. Additionally, as a consequence of the energy estimate,  
\begin{equation}\label{e2}
\|u\|_{C([0,\tau ],H_0^1(\ring{M} ))}+\|\partial _t u\|_{C([0,\tau ],L^2((\ring{M}  ))}\leq C\|(u_0,u_1)\|_{\mathcal{H}}.
\end{equation}
Here $C=C(\|q\|_\infty +\|a\|_\infty )$ is a nondecreasing function.

\smallskip
Denote by $\nu$ the  unit normal vector field pointing inward $\ring{M}$ and set $\partial _\nu u =\langle \nabla u,\nu\rangle$. From \cite[Lemma 2.4.1]{BY}, $\partial _\nu u\in L^2(B)$ and
\begin{equation}\label{e3}
\|\partial_\nu u\|_{L^2(B)}\le c_M\left( \|(u_0,u_1)\|_{\mathcal{H}} +\|qu + a\partial_t u\|_{L^1((0,\tau) , L^2(\ring{M})} \right),
\end{equation}
where $c_M$ is a constant depending only on $M$.

\smallskip
A combination of \eqref{e2} and \eqref{e3} yields
\begin{equation}\label{e4}
\|\partial_\nu u\|_{L^2(B)}\le C\|(u_0,u_1)\|_{\mathcal{H}},
\end{equation}
with a constant $C$ of the same form as in \eqref{e2}.

\smallskip
Pick $\Gamma$, a non empty open subset of $\partial M$, and $\tau >0$. In all of this paper, we assume that  $(\Gamma ,\tau)$ geometrically control $M$. We refer to \cite{Le} for a precise definition of this assumption. It is worth mentioning that the notion of geometric control was introduced by Bardos, Lebeau and Rauch in \cite{BLR}. 

\smallskip
Fix $(q_0,a_0)\in L^\infty (\ring{M})\times L^\infty (\ring{M})$. In light of \cite[theorem page 169]{Le} (which remains valid for the wave operator plus an operator involving space derivatives of first order) and bearing in mind that controllability is equivalent to observability, we can state the following inequality
\begin{equation}\label{e5}
2\kappa \| (u_0,u_1)\|_{\mathcal{H}} \le \| \partial _\nu \mathcal{S}_{q_0,a_0}(u_0,u_1)\|_{L^2(\Gamma \times (0,\tau ))},
\end{equation}
for some constant $\kappa >0$.

\smallskip
By the perturbation argument  in \cite[Proposition 6.3.3, page 189]{TW}, we assert that there exists $\beta >0$ so that, for any $(q,a)=(q_0,a_0)+(\widetilde{q},\widetilde{a})$, with $(\widetilde{q},\widetilde{a})\in W^{1,\infty}(\ring{M} )\oplus L^\infty (\ring{M})$ satisfies $\|(\widetilde{q},\widetilde{a})\|_{W^{1,\infty}(\ring{M} )\oplus L^\infty (\ring{M})}\le \beta$,
\begin{equation}\label{e6}
\kappa \| (u_0,u_1)\|_{\mathcal{H}} \le \| \partial _\nu \mathcal{S}_{q,a}(u_0,u_1)\|_{L^2(\Gamma \times (0,\tau ))}.
\end{equation}
Here $\kappa >0$ is the constant in \eqref{e5}.

\smallskip
Set
\[
\mathcal{D}=\{ (q,a)=(q_0,a_0)+(\widetilde{q},\widetilde{a});\; (\widetilde{q},\widetilde{a})\in W^{1,\infty}(\ring{M} )\oplus L^\infty (\ring{M})\; \mbox{and}\; \|(\widetilde{q},\widetilde{a})\|_{W^{1,\infty}(\ring{M} )\oplus L^\infty (\ring{M})}\le \beta \}.
\]
When $q_0\ge 0$, we will use $\mathcal{D}^+$ instead of $\mathcal{D}$.

\begin{remark}
The constant $\kappa$ in \eqref{e5} is obtained in an abstract way from the HUM method and therefore it is not possible to derive how it depends on $q_0$ and $a_0$. This explains why we used a perturbation argument in order to get that  the observability constant in \eqref{e6} is uniform in $(q,a)\in \mathcal{D}$.
\end{remark}

Define the initial-to-boundary operator $\Lambda_{q,a}$ as follows
\[
\Lambda_{q,a}: (u_0,u_1)\in \mathcal{H} \mapsto \partial_\nu \mathcal{S}_{q,a}(u_0,u_1)\in L^2(\Gamma \times (0,\tau )).
\] 

Let $\mathcal{H}_1=(H_0^1(\ring{M})\cap H^2(\ring{M}))\oplus H_0^1(\ring{M})$. In light of the fact that 
\[
\partial _t\mathcal{S}_{q,a}(u_0,u_1)=\mathcal{S}_{q,a}(\Delta u_0 - q u_0 -au_1,u_1),
\]
one can easily obtain that $\Lambda _{q,a}\in \mathscr{B}(\mathcal{H}_1,H^1((0,\tau ),L^2(\Gamma ))$. Additionally, as a consequence of \eqref{e4},
\[
\|\Lambda_{q,a}\|_{\mathscr{B}(\mathcal{H}_1,H^1((0,\tau ),L^2(\Gamma )))}\le C.
\]
Here $C=C(\|q\|_\infty +\|a\|_\infty )$ is a nondecreasing function.

\smallskip
Our main purpose is the stability issue for the inverse problem of recovering $(q,a)$ from the initial-to-boundary operator $\Lambda_{q,a}$. We provide a method based on the spectral analysis of the unbounded operator defined on $H_0^1(\ring{M})\oplus L^2(\ring{M})$ by
\[ 
\mathcal{A}_{q,a}=\left( 
\begin{array}{cc}
0 & I  \\
\Delta -q  & -a  \\
 \end{array} 
 \right),
 \]
 with domain $D(\mathcal{A}_{q,a})=\mathcal{H}_1$.
 
\smallskip
When $q\ge 0$ and $a=0$, $-i\mathcal{A}_{q,0}$ is self-adjoint and $\mathcal{A}_{q,0}^{-1}$ is compact and therefore $\mathcal{A}_{q,0}$ is diagonalizable. In that case, under the assumption that the Hardy inequality holds in $M$, we prove in Theorem \ref{theorem1.1} below H\"older stability estimate with exponent $1/2$. In the general case $-i\mathcal{A}_{q,a}$ is no longer self-adjoint but it is a ``nice'' perturbation of  the self-adjoint operator $-i\mathcal{A}_{0,0}$. This observation enables us to show that  $\mathcal{A}_{q,a}$ possesses Riesz basis consisting of eigenfunctions. This result, combined with a fine analysis of the behavior of an eigenfunction near its zeroes, enables us to obtain in Theorem \ref{theorem3.1} below  H\"older stability estimate with some indefinite exponent.

\smallskip
These kind of inverse problems were initiated by the first and the second authors in \cite{AC1}. The main idea in \cite{AC1}  combine both the stability estimate for inverse source problems by \cite{ASTT} and the spectral decomposition of the unbounded operator associated with the IBVP under consideration. In the present work we improve the logarithmic stability estimates presented in \cite{AC1}. We actually prove that the stability is of H\"older type. The new ingredient we used consists, roughly speaking, in quantifying globally, the property saying that a non zero solution of an elliptic equation can not have a zero of infinite order.

\smallskip
It is worth mentioning that the geometric control condition on $\Gamma$ can be removed. But it that case the stability is of (at most) logarithmic type. We refer to the recent paper \cite{ACT} by authors  for more details.

\smallskip
A neighbor inverse problem which is much harder to tackle is the one consisting in the determination of the damping boundary coefficient from the corresponding initial-to-boundary operator. First results for this problem was obtained by the first and the second authors in \cite{AC2} and \cite{AC3}.


\section{Stability around a zero damping coefficient}

In the present section we assume that $M$  is embedded in a n-dimensional complete manifold without boundary $N=(N,g)$ and the following Hardy's inequality is fulfilled
\begin{equation}\label{1.3}
\int_M|\nabla f(x)|^2dV \ge c\int_M \frac{|f(x)|^2}{d(x,\partial M)^2}dV,\;\; f\in H_0^1(\ring{M}),
\end{equation}
for some constant $c>0$, where $dV$ is the volume form on $M$, $d$ is the geodesic distance and $d(\cdot ,\partial M)$ is the distance to $\partial M$.

\smallskip
Define $r_x(v)=\inf \{|t|;\; \gamma_{x,v}(t)\not\in \ring{M}\}$, where $\gamma_{x,v}$ is the geodesic satisfying the initial condition $\gamma_{x,v}=x$ and $\dot{\gamma}_{x,v}=v$. As it is observed in \cite{Ra}, the hardy inequality holds for $M$ whenever $M$ has the following uniform interior cone property: there are an angle $\alpha >0$ and a constant $c_0>0$ so that, for any $x\in M$, there exists an $\alpha$-angled cone $C_x\subset T_x$ with the property that $r_x(v)\le c_0d(x,\partial M)$, for all $v\in C_x$. The proof of this result follows the method by Davies \cite[page 25]{Da} for the flat case. We mention that Hardy's inequality holds for any bounded Lipschitz domain of $\mathbb{R}^n$ with constant $c\le \frac{1}{4}$, with equality when $\Omega$ is convex.

\smallskip
For $m>0$, set $\mathcal{D}_m=\{(q,a)\in \mathcal{D}^+\cap H^2(\ring{M})\oplus H^2(\ring{M});\; \|(q,a)\|_{H^2(\ring{M})\oplus H^2(\ring{M})}\le m\}$. Let $m_0$ be sufficiently large so that $\mathcal{D}_m\ne \emptyset$, for all $m\ge m_0$. In the sequel $m\ge m_0$ will be fixed.

\begin{theorem}\label{theorem1.1}
Let $(q,0)\in \mathcal{D}_m$ with $q\in C^1(M)$ and $q\ge 0$. Then there exists a constant $C>0$, that can depend on the data and $q$, so that, for any $(\widetilde{q},\widetilde{a})\in \mathcal{D}_m$, we have
\[
\| \widetilde{q}-q\|_{L^2(\ring{M})}+\|\widetilde{a}-0\|_{L^2(\ring{M})}\le C\|\Lambda_{\widetilde{q},\widetilde{a}}-\Lambda_{q,0}\|^{1/2}.
\]
Here $\|\Lambda_{\widetilde{q},\widetilde{a}}-\Lambda_{q,0}\|$ denotes the norm of $\Lambda_{\widetilde{q},\widetilde{a}}-\Lambda_{q,0}$ in $\mathscr{B}(\mathcal{H}_1,H^1((0,\tau ),L^2(\Gamma )))$.
\end{theorem}

We firstly establish a weighted interpolation inequality. To do so, we will use Hopf's maximum principle that we recall in the sequel.

\begin{lemma}
(Hopf's maximum principle) Let $q\in C(M)$ and  $u\in C^1(M)\cap C^2(\ring{M})\cap H_0^1(M)$ satisfying $q\le 0$ and $\Delta u+qu\le 0$. If $u$ is non identically equal to zero, then $u>0$ in $\ring{M}$ and $\partial _\nu u(y)=\langle \nabla u(y),\nu (y)\rangle >0$ for any $y\in \partial M$.
\end{lemma}

\begin{proof}
Follows the same lines to that of \cite[Lemma 3.4, page 34 and Theorem 3.5, page 35]{GT}. The tangent ball in the classical Hopf's lemma is substitute by a tangent geodesic ball (see the construction in \cite[Proof of Theorem 9.2, page 51]{PS}).
\end{proof}

\begin{proposition}\label{proposition1}
Let $q\in C(M)$ and  $u\in C^1(M)\cap C^2(\ring{M})\cap H_0^1(M)$ satisfying $q\le 0$ and $\Delta u+qu\le 0$.
If $u$ is  non identically equal to zero, then
\[
u(x)\ge c_ud(x,\partial M),\;\; x\in M,
\]
where $c_u$ is a constant that can depend on $u$ and $M$.
\end{proposition}

\begin{proof}
Let $0<\epsilon $ to be specified later. Let $x\in M$ so that $d(x,\partial M)\le \epsilon$ and $y\in \partial M$ satisfying $d(x,\partial M )=d(x,y)$. Since $N$ is complete, there exist a unit speed minimizing geodesic $\gamma : [0,r]\rightarrow M$ such that $\gamma (0)=y$, $\gamma (r)=x$ and $\dot{\gamma}(0)=\nu (y)$, where we set $r=d(x,\partial M)$ (see for instance \cite[page 150]{Pe}).

\smallskip
Define $\phi(t)=u(\gamma (t))$. Then
\begin{align*}
&\phi '(t)=du(\gamma (t))(\dot{\gamma}(t))
\\
&\phi''(t)=d^2u(\gamma (t))(\dot{\gamma}(t),\dot{\gamma}(t))+du(\gamma (t))(\ddot{\gamma}(t)).
\end{align*}
Here $\dot{\gamma}(t)=\dot{\gamma}^i(t)\partial _i\in T_{\gamma (t)}$. Observe that by the geodesic equation
\[
\ddot{\gamma}^k(t)=-\dot{\gamma}^i(t)\dot{\gamma}^j(t)\Gamma_{ij}^k(\gamma (t)),
\]
where $\Gamma_{ij}^k$ are the Christoffel symbols associated to the metric $g$.

\smallskip
Taking into account that $\phi '(0)=du(y)(\nu (y))=\langle \nabla u(y),\nu (y)\rangle =\partial _\nu u(y)$, we get
\[
\phi (r)=r\partial _\nu u(y)+\frac{r^2}{2}\phi ''(st),
\]
for some $0<s<1$. Hence, there exist $c>0$ depending on $u$ and $M$ so that
\[
\phi(r)\ge 2r\eta -cr^2\ge r\eta + r(\eta -c\epsilon)
\]
with $2\eta = \min_{y\in\Gamma} \partial _\nu u(y) >0$ (by the compactness of $\Gamma$). Consequently,
\[
\phi (r)\ge r\eta 
\]
provided that $\epsilon \le \eta /c$. In other words, we proved
\begin{equation}\label{4}
u(x)=\phi (r )\ge r \eta =\eta d(x,\partial M ).
\end{equation}

On the other hand, an elementary compactness argument yields, where  $M^\epsilon =\{x\in M;\; d(x,\Gamma )\ge \epsilon\}$, 
\begin{equation}\label{6}
u(x)\ge \min_{z\in M^\epsilon}u(z)\ge \frac{\min_{z\in M^\epsilon}u(z)}{\max_{z\in M^\epsilon}d(z,\partial M )}d(x,\partial M),\;\; x\in M^\epsilon .
\end{equation}
In light of \eqref{4} and \eqref{6}, we end up getting
\[
u(x)\ge c_u d(x,\partial M ),\;\; x\in M.
\]
\end{proof}

As a consequence of Proposition \ref{proposition1} and Hardy's inequality, we have
\begin{corollary}\label{corollary1}
(Weighted interpolation inequality) Let $q\in C(M)$, $q\le 0$, and $u\in C^2(M) \cap H_0^1(M)$ non identically equal to zero satisfying $\Delta u+qu\le 0$. There exists a constant $c_u$, that can depend only on $u$ and $M$ so that, for any $f\in H^2(\ring{M})$,
\[
\|f\|_{L^2(\ring{M})}\le c_u\|fu\|_{L^2(\ring{M})}^{1/2}\|f\|_{H^2(\ring{M})}^{1/2}.
\]
\end{corollary}

\begin{proof}
By Proposition \ref{proposition1}, $u(x)\ge c_ud(x,\partial M )$. Therefore
\[
\int_{\ring{M}} f(x)^2dV(x)\le c_u^{-1} \int_{\ring{M}}\frac{f(x)^2u(x)^2}{d(x,\partial M)^2}dV(x).
\]
Combined with Hardy's inequality \eqref{1.3}, this estimate gives
\begin{equation}\label{10}
\int_{\ring{M}} f(x)^2dV(x)\le c_u^{-1}c \int_{\ring{M}}|\nabla (fu)(x)|^2dV(x).
\end{equation}
But, from usual interpolation inequalities,
\[
\|fu\|_{H^1(\ring{M})}\le C\| fu\|_{L^2(\ring{M})}^{1/2}\|fu\|_{H^2(\ring{M})}^{1/2},
\]
where the constant $C$ depends only on $M$.

Consequently, \eqref{10} implies
\[
\|f\|_{L^2(\ring{M})}\le c_u\|fu\|_{L^2(\ring{M})}^{1/2}\|f\|_{H^2(\ring{M})}^{1/2},
\]
which is the expected inequality
\end{proof}

Fix $0\le q\in C^1 (M)$ and consider the unbounded operator $A=-\Delta+q$, with domain $D(A)=H^2(\ring{M}) \cap H_0^1(\ring{M})$. An extension of \cite[Theorem 8.38, page 214]{GT} to a compact Riemannian manifold with boundary shows that the first eigenvalue of $A$, denoted by $\lambda _1$ is simple and has a positive eigenfunction. Let then $\phi _1\in C^2 (M)$ (by elliptic regularity) be the unique first eigenfunction satisfying $\phi_1 >0$ and normalized by $\|\phi _1\|_{L^2(\ring{M})}=1$. Since $\Delta \phi _1-qu =-\lambda _1\phi_1$, the Hopf's maximum principle is applicable for $\phi_1$. Therefore, a particular weight in the preceding corollary is obtained by taking $u=\phi _1$.

\begin{corollary}\label{corollary2}
There exists a constant $c>0$, that can depend on $\phi_1$, so that, for any $f\in H^2(\ring{M} )$,
\[
\|f\|_{L^2(\ring{M})}\le c\|f\phi _1\|_{L^2(\ring{M})}^{1/2}\|f\|_{H^2(\ring{M})}^{1/2}.
\]
\end{corollary}

\begin{proof}[Completion of the proof of Theorem \ref{theorem1.1}]

Let $(q,0), (\widetilde{q},\widetilde{a}) \in \mathcal{D}_m$ with $q\in C^1(M)$ and $q\ge 0$. Denote by $0\le \phi_1$  the first eigenfunction of $-\Delta +q$ under Dirichlet boundary condition, normalized by $\|\phi_1\|_{L^2(\ring{M})}=1$.  Set $u=\mathcal{S}_{q,0}(\phi_1,i\sqrt{\lambda _1}\phi_1)=e^{i\sqrt{\lambda _1}t}\phi_1$ and $\widetilde{u}=\mathcal{S}_{\widetilde{q},\widetilde{a}}(\phi_1,i\sqrt{\lambda _1}\phi_1)$. Then 
\[
v=\mathcal{S}_{\widetilde{q},\widetilde{a}}(\phi_1,i\sqrt{\lambda _1}\phi_1)-\mathcal{S}_{q,0}(\phi_1,i\sqrt{\lambda _1}\phi_1)
\] 
is the solution of the following IBVP
\begin{equation}\label{3.1b}
\left\{
\begin{array}{lll}
 \partial _t^2 v - \Delta v + \widetilde{q}v + \widetilde{a}(x) \partial_t v = -[(\widetilde{q}-q)+i\sqrt{\lambda _1}\widetilde{a}]e^{i\sqrt{\lambda _1}t}\phi_1 \;\; &\mbox{in}\;   Q=\ring{M} \times (0,\tau), 
 \\
v = 0 &\mbox{on}\;  B =\partial M \times (0,\tau), 
\\
v(\cdot ,0) = 0,\; \partial_t v (\cdot ,0) = 0.
\end{array}
\right.
\end{equation}
From \cite[Corollary 2.1]{AC1}
\begin{align*}
\| \phi _1(\widetilde{q}-q)\|_{L^2(\ring{M})}+\|\phi_1\widetilde{a}\|_{L^2(\ring{M})}&\le c_\beta \|\partial _\nu v\|_{H^1((0,\tau ),L^2(\Gamma ))}
\\
&\le c_\beta\|\Lambda_{\widetilde{q},\widetilde{a}}-\Lambda_{q,0}\| .
\end{align*}
This inequality, combined with Corollary \ref{corollary2}, yields
\[
\|\widetilde{q}-q\|_{L^2(\ring{M})}+\|\widetilde{a}-0\|_{L^2(\ring{M})}\le C\|\Lambda_{\widetilde{q},\widetilde{a}}-\Lambda_{q,0}\|^{1/2}.
\]
The proof is then complete.
\end{proof}

\section{Stability for the general case}

We aim to establish the following theorem.

\begin{theorem}\label{theorem3.1}
Let $(q,a)\in \mathcal{D}$. There exist $C>0$ and $0<\alpha <1$, that can depend on $(q,a)$, so that for any $(\widetilde{q},\widetilde{a})\in \mathcal{D}$, we have
\[
\|\widetilde{q}-q\|_{L^2(\ring{M})}+\|\widetilde{a}-a\|_{L^2(\ring{M})}\le C\|\Lambda_{\widetilde{q},\widetilde{a}}-\Lambda_{q,a}\|^\alpha .
\]
Here $\|\Lambda_{\widetilde{q},\widetilde{a}}-\Lambda_{q,a}\|$ denotes the norm of $\Lambda_{\widetilde{q},\widetilde{a}}-\Lambda_{q,a}$ in $\mathscr{B}(\mathcal{H}_1,H^1((0,\tau ),L^2(\Gamma )))$.
\end{theorem}

Prior to proving this theorem, we make the spectral analysis of the operator $\mathcal{A}_{q,a}$. Denote the sequence of eigenvalues, counted according to their multiplicity, of $A=-\Delta$, with domain $D(A)=H^2(\ring{M} )\cap H_0^1(\ring{M} )$, by $0<\lambda _1<\lambda _2\leq \ldots \lambda _k \leq \ldots$.
 
\smallskip
Consider the unbounded operators defined, on $\mathcal{H}=H_0^1(\ring{M} )\oplus L^2(\ring{M})$, by
\[ 
\mathcal{A}_0=\left( 
\begin{array}{cc}
0 & I  \\
-A  & 0  \\
 \end{array} 
 \right),\;\; D(\mathcal{A}_0)=\mathcal{H}_1=\left[H^2(\ring{M} )\cap H_0^1(\ring{M} )\right]\oplus H_0^1(\ring{M})
 \]
 and $\mathcal{A}_{q,a}=\mathcal{A}_0+\mathcal{B}_{q,a}$ with $D(\mathcal{A}_{q,a})=D(\mathcal{A}_0)$, where
 \[ 
\mathcal{B}_a=\left( 
\begin{array}{cc}
0 & 0 \\
-q & -a  \\
 \end{array} 
 \right) \in \mathscr{B}(\mathcal{H}).
 \]
 
According to \cite[Proposition 3.7.6, page 100]{TW}, we know that $\mathcal{A}_0$ is skew-adjoint operator with $0\in \rho (\mathcal{A}_0 )$ and
\[ 
\mathcal{A}_0^{-1}=\left( 
\begin{array}{cc}
0 & -A^{-1}  \\
I  & 0  \\
 \end{array} 
 \right).
 \]
We note that, since $\mathcal{A}_0^{-1}:\mathcal{H} \rightarrow \mathcal{H}_1$ is bounded and the embedding $\mathcal{H}_1\hookrightarrow \mathcal{H}$ is compact, $\mathcal{A}_0^{-1}:\mathcal{H} \rightarrow \mathcal{H}$ is compact.

\smallskip
Also, from \cite[Proposition 3.7.6, page 100]{TW}, $\mathcal{A}_0$ is diagonalizable and its spectrum consists in the sequence $(i\sqrt{\lambda _k})$.

\smallskip
Consider the bounded operator $\mathcal{C}_{q,a}=(i\mathcal{A}_0^{-1})(-i\mathcal{B}_{q,a})(i\mathcal{A}_0^{-1})$. Let $s_k(\mathcal{C}_{q,a})$ be the singular values of $\mathcal{C}_{q,a}$, that is the eigenvalues of $(\mathcal{C}_{q,a}^\ast \mathcal{C}_{q,a})^{1/2}$. In light of \cite[formulas (2.2) and (2.3), page 27]{GK}, we have
\[
s_k(\mathcal{C}_{q,a})\le \|\mathcal{B}_{q,a}\|s_k(i\mathcal{A}_0^{-1})^2=\|\mathcal{B}_{q,a}\|\lambda _k^{-1} ,
\]
where $\|\mathcal{B}_{q,a}\|$ denote the norm of $\mathcal{B}_{q,a}$ in $\mathscr{B}(\mathcal{H})$.

\smallskip
On the other hand, referring to Weyl's asymptotic formula, we have $\lambda _k =O(k^{-2/n})$. Consequently, $\mathcal{C}_{q,a}$ belongs to the Shatten class $\mathcal{S}_p$ for any $p>n/2$, that is
\[
\sum_{k\ge 1}\left[s_k(\mathcal{C}_{q,a})\right]^p<\infty .
\]
We apply \cite[Theorem 10.1, page 276]{GK} in order to get that the spectrum of $\mathcal{A}_{q,a}$ consists in a sequence of eigenvalues $(\mu_{q,a,k})$, counted according to their multiplicity, and the corresponding eigenfunctions $(\phi_{q,a,k})$ form a Riesz basis of $\mathcal{H}$.

\smallskip
Fix $(q,a,k)$ and set $\mu  =\mu _{q,a,k}$ and $\phi= \phi_{q,a,k} =(\varphi  ,\psi )\in \mathcal{H}_1$ be an eigenfunction associated to $\mu$. Then it is straightforward to check that
$\psi =\mu \varphi$ and $(-\Delta +q+a\mu +\mu ^2)\varphi=0$ in $\ring{M}$. Since $-\Delta \varphi =f$ in $\ring{M}$ with $f= (q+a\mu +\mu ^2)\varphi$, we can use iteratively \cite[Corollary 7.11, page 158]{GT} (Sobelev embedding theorem) together with \cite[Theorem 9.15, page 241]{GT} in order to obtain that $\varphi \in W^{2,p}(\ring{M})$ for any $1<p<\infty$. In particular $\varphi ,|\varphi |^2\in W^{2,n}(\ring{M})\cap C^0(M)$. This property of $\varphi$ will be used in the proof of Proposition \ref{proposition3.1} below.

\smallskip
The following result enters in an essential way in the proof of the weighted interpolation inequality that we will use to prove Theorem \ref{theorem3.1}.
\begin{proposition}\label{proposition3.1}
There exists $\delta >0$ so that $\varphi ^{-\delta} \in L^1(\ring{M})$. 
\end{proposition}

The proof of this proposition is given in the end of this section.

\begin{lemma}\label{lemma3.1}
(Weighted interpolation inequality)
There exists a constant $C$, that can depend on $\varphi$, so that for any $f\in L^\infty (\ring{M})$, we have
\[
\|f\|_{L^2(\ring{M})}\le C\|f\|_{L^\infty (\ring{M})}^{\frac{2}{2+\delta}}\|f\varphi\|_{L^2(\ring{M})}^{\frac{\delta}{2+\delta}}.
\]
Here $\delta$ is as in Proposition \ref{proposition3.1}.
\end{lemma}

\begin{proof}
Set $p=\frac{2}{\delta}$ and $\alpha =\frac{2}{p}=\frac{2\delta}{2+\delta}$. Therefore, the exponent conjugate to $p$, $p^\ast =\frac{2+\delta}{2}$ and $\alpha p^\ast=\delta$. We get by applying H\"older's inequality
\[
\int_{\ring{M}}|f|^\alpha dV=\int |f\varphi |^\alpha |\varphi ^{-\alpha}|dV\le \||f\varphi |^\alpha\|_{L^p(\ring{M})}\||\varphi ^{-\alpha}|\|_{L^{p^\ast}(\ring{M})}=\|f\varphi\|_{L^2(\ring{M})}^{\frac{2}{p}}\|\varphi ^{-\delta}\|_{L^1(\ring{M})}^{1/p^\ast}.
\]
Whence
\begin{equation}\label{3.1}
\|f\|_{L^\alpha (\ring{M})}\le \|f\varphi\|_{L^2(\ring{M})}\|\varphi ^{-\delta}\|_{L^1(\ring{M})}^{1/\delta}.
\end{equation}
On the other hand
\begin{equation}\label{3.2}
\|f\|_{L^2(\ring{M})}\le \|f\|_{L^\infty (\ring{M})}^{\frac{2-\alpha}{2}}\|f\|_{L^\alpha (\ring{M})}^{\frac{\alpha}{2}}.
\end{equation}
A combination of \eqref{3.1} and \eqref{3.2} yields
\[
\|f\|_{L^2(\ring{M})}\le C\|f\|_{L^\infty (\ring{M})}^{\frac{2}{2+\delta}}\|f\varphi\|_{L^2(\ring{M})}^{\frac{\delta}{2+\delta}}
\]
which is the expected inequality.
\end{proof}

We are now ready to complete the proof of Theorem \ref{theorem3.1}. Similarly as in the proof of the completion of Theorem \ref{theorem1.1}, we have, by taking $(u_0,u_1)=\phi_{q,a,k}$,
\[
\|\varphi (\widetilde{q}-q)\|_{L^2(\ring{M})}+\|\varphi (\widetilde{a}-a)\|_{L^2(\ring{M})}\le C\|\Lambda_{\widetilde{q},\widetilde{a}}-\Lambda_{q,a}\|.
\]
According to the weighted interpolation inequality in Lemma \ref{lemma3.1}, this inequality entails
\[
\|\widetilde{q}-q\|_{L^2(\ring{M})}+\|\widetilde{a}-a\|_{L^2(\ring{M})}\le C\|\Lambda_{\widetilde{q},\widetilde{a}}-\Lambda_{q,a}\|^{\frac{\delta}{2+\delta}}.
\]

\begin{proof}[Proof of Proposition \ref{proposition3.1}]

{\it First step.} Denote by $B$ the unit ball of $\mathbb{R}^n$ and let $B_+=B\cap \mathbb{R}_+^n$, with $\mathbb{R}_+^n=\{ x=(x',x_n)\in \mathbb{R}^n;\; x_n >0\}$. Let $L$ be a second order differential operator acting as follows
\[
Lu=\partial _j(a_{ij}\partial _i u )+ C\cdot \nabla u+du.
\]
Assume that $(a_{ij})$ is a symmetric matrix with entries in
$C^1(2\overline{B_+} )$, $C\in L^\infty (2B_+)^n$ is real
valued 
and $d\in L^\infty (2B_+)$ is complex valued. Assume moreover that
\[
a_{ij}(x)\xi _j\cdot \xi_j\ge \kappa_0 |\xi |^2,\;\; x\in 2B_+,\; \xi \in \mathbb{R}^n,
\]
for some $\kappa_0 >0$.

\smallskip
Let $u\in W^{2,n}(2B_+)\cap C^0(2\overline{B_+})$ be a weak solution of $Lu=0$
satisfying $u=0$ on $\partial (2B_+)\cap \overline{\mathbb{R}^n_+}$ and $|u|^2\in W^{2,n}(2B_+)\cap C^0(2\overline{B_+})$.
\smallskip
From \cite[Theorem 1.1, page 942]{AE}, there exists a constant $C$, that can
depend on $u$, so that the following doubling inequality at the boundary
\begin{equation*}
\int_{B_{2r}\cap B_+}{|u|}^2dx \le C\int_{B_r\cap B_+}{|u|}^2dx,
\end{equation*}
holds for any ball $B_{2r}$, of radius $2r$, contained in $2B$.

\smallskip
On the other hand  simple calculations yield, where $v= \Re u$ and $w=\Im u$,
\[
\partial _j(a_{ij}\partial _i |u|^2 )+ 2C\cdot \nabla |u|^2+4(|\Re d|+|\Im d|)|u|^2\ge 2a_{ij}\partial_iv\partial_jv+2a_{ij}\partial_iw\partial_jw \ge 0\;\; \mbox{in}\; 2B_+
\]
and $|u|^2=0$ on $\partial (2B_+)\cap
\overline{\mathbb{R}^n_+}$.

\smallskip
Harnak's inequality at the boundary (see \cite[Theorem 9.26, page 250]{GT}) entails 
\begin{equation*}
\sup_{B_r \cap B_+}{|u|}^2\le \frac{C}{|B_{2r}|}\int_{B_{2r}\cap B_+}
{|u|}^2dx,
\end{equation*}
for any ball $B_{2r}$, of radius $2r$, contained in $2B$.

\smallskip
Define $\widetilde{u}$  by
\[
\widetilde{u} (x',x_n)=u(x',x_n)\;\; \mbox{if}\; 
(x',x_n)\in 2B_+ ,\quad \widetilde{u} (x',x_n)=u(x',-x_n)\; \mbox{if}\; (x',-x_n)\in 2B_+.
\]
Therefore  $\widetilde{u} $ belongs  to $H^1(2B)\cap L^\infty(2B)$ and satisfies 
\begin{eqnarray}
\int_{B_{2r}}|\widetilde{u}|^2dx \le C\int_{B_r}|\widetilde{u}|^2dx,\label{3.3}\\
\sup_{B_r}|\widetilde{u}|^2\le \frac{C}{|B_{2r}|}\int_{B_{2r}}
{|\widetilde u|}^2dx,\label{3.4+}
\end{eqnarray}
 for any ball $B_{2r}$, of radius $2r$, contained in $2B$.

\smallskip
Inequalities \eqref{3.3} and \eqref{3.4+} at hand, we mimic the proof of \cite[Theorem 4.2, page 1784]{CT} in order to obtain that $\widetilde{u}^{-\delta}\in L^1(B)$ for some $\delta >0$, that can depend on $u$. Whence $u^{-\delta}\in L^1(B_+)$.

\smallskip
{\it Second step.} As $\partial M$ is compact, there exists a finite cover $(U_\alpha )$ of $\partial M$  and  $C^\infty$-diffeomorphisms  $f_\alpha: U_\alpha \rightarrow 2B$ so that $f_\alpha (U_\alpha \cap \ring{M})=2B_+$, $f_\alpha (U_\alpha \cap \partial M)=2B \cap \overline{\mathbb{R}^n_+}$ and, for any $x\in \partial M$, $x\in V_\alpha =f_\alpha^{-1}(B)$, for some $\alpha$. Then $u_\alpha =\varphi \circ f_\alpha ^{-1}$ satisfies $L_\alpha u_\alpha =0$ in $2B$ and $u=0$ on $\partial (2B_+)\cap \overline{\mathbb{R}^n_+}$ for some $L=L_\alpha$ obeying to the conditions of the first step. Hence $u_\alpha ^{-\delta _\alpha}\in L^1(B_+)$ and then $\varphi ^{-\delta _\alpha}\in L^1(V_\alpha )$. Let $V$ the union of $V_\alpha$'s. Since $u\in L^\infty (V)$, which is a consequence of \eqref{3.4+}, $u^{-\delta_0}\in L^1(V)$ with $\delta_0=\min \delta_\alpha$. Next, let $\epsilon$ sufficiently small in such a way that $M\setminus M_\epsilon\subset V$, where $M=\{ x\in M;\; \mbox{dist}(x,\partial M)>\epsilon \}$. Proceeding as previously  it is not hard to get that there exists $\delta _1$ so that $\varphi ^{-\delta _1}\in L^1(M_{\epsilon/2})$. Finally, as it is expected, we derive that $\varphi ^{-\delta}\in L^1(\ring{M} )$ with $\delta =\min (\delta_0,\delta_1)$.

\end{proof}

\end{document}